\documentclass[11.5pt]{article}
\usepackage{amsmath,amsthm}
\usepackage[authoryear]{natbib}
\usepackage{graphicx}
\usepackage{verbatim}
\usepackage{hyperref}
\usepackage{multirow}
\usepackage{times}
\usepackage{color}
\newtheorem{thm}{Theorem} 
 
\newtheorem{cor}[thm]{Corollary}

\newtheorem{definition}{Definition}[section]

\newtheorem{condition-t}{Translated Condition}

\theoremstyle{remark}

\newcommand{\bed}{\begin{definition}}
\newcommand{\eed}{\end{definition}}
\newcommand{\beq}{\begin{equation}}
\newcommand{\eeq}{\end{equation}}

\newcommand{\bitem}{\begin{itemize}}
\newcommand{\eitem}{\end{itemize}}

\newcommand{\beqn}{\begin{equation}}
\newcommand{\eeqn}{\end{equation}}
\newcommand{\balign}{\begin{align}}
\newcommand{\ealign}{\end{align}}

\usepackage{amssymb}

\title{Concentration Inequalities for Incomplete U-statistics over Arbitrary Sampling Graphs}
\author{Zheng Tracy Ke\\Harvard University}

\begin{document} 

\maketitle

\begin{abstract}
Let $X_1, X_2, \ldots, X_n$ be independent random vectors. For a directed graph 
$G=(V,E)$ with vertex set $V=\{1,2,\ldots,n\}$ and a collection of bivariate 
kernels $\{h_e:e\in E\}$, we consider
\[
U=\sum_{e=(i,j)\in E} h_e(X_i,X_j).
\]
This framework generalizes incomplete U-statistics by allowing the random vectors 
to be non-identically distributed, the kernels to be asymmetric and edge-dependent, 
and the sampling structure to be specified by an arbitrary graph.

We derive several concentration inequalities for $U-\mathbb{E}U$. The main proof 
strategy exploits edge-coloring results from graph theory and relates the tail 
behavior of $U$ to the chromatic index of $G$. This approach is elementary, 
transparent, and readily adaptable to broader settings, including U-statistics of 
order $m>2$ and statistics involving doubly indexed random vectors.
\end{abstract}

\section{Introduction} \label{sec:Intro}
Classical U-statistics \citep{hoeffding1948class} play a fundamental role in a wide range of statistical problems. Let $X_1, \ldots, X_n \in \mathbb{R}^p$ be independent and identically distributed (i.i.d.) random vectors and $h:\mathbb{R}^p \times\mathbb{R}^p \to \mathbb{R}$ a symmetric bivariate kernel. An order-$2$ U-statistic is defined as 
\beq \label{classical-U-stat}
U^{*}:=\sum_{1\leq i<j \leq n} h(X_i, X_j).
\eeq
The normalized version, $U^*/\binom{n}{2}$, is commonly used to estimate $\theta := \mathbb{E}[h(X_1, X_2)]$. 

Computing a classical U-statistic requires evaluating the kernel value on $\binom{n}{2}=O(n^2)$ subsets, which can be computationally expensive when $n$ is extremely large. To alleviate this burden, \citet{blom1976incomplete} introduced incomplete U-statistics, which are constructed by randomly subsampling a subset of the pairs appearing in the full U-statistics.

In this paper, we consider a class of generalized incomplete U-statistics. Suppose  $X_1, X_2, \ldots, X_n$  are independent but not necessarily identically distributed. Let $G=(V,E)$ be an undirected graph with the node set $V=\{1,2,\ldots,n\}$, and let $\{h_e\}_{e\in E}$ be a collection of bivariate kernels indexed by the edges in $G$. 
These kernels are not necessarily symmetric, i.e., $h_e(x,y)\neq h_e(y,x)$ is allowed. Define
\beq \label{our-U-stat}
U := \sum_{e=(i,j)\in E}h_e(X_i, X_j). 
\eeq
The classical U-statistics and the incomplete U-statistics are special cases where $G$ is the complete graph and the Erd\H{o}s--R\'enyi random graph, respectively. 
Our setting is considerably more general, allowing $G$ to be an arbitrary deterministic graph.
Moreover, we relax the i.i.d. assumption on the observations $X_i$ and the homogeneity and symmetry assumptions on the kernels.  

We aim to obtain concentration inequalities for $U-\mathbb{E}U$. Such results have potential applications to the analysis of design-based incomplete U-statistics \citep{brown1978reduced, rempala2003incomplete, rempala2004minimum, kong2021design}, covariate-assisted Bradley-Terry models \citep{turner2012bradley,fan2024uncertainty} with arbitrary comparison graphs, and partially observed latent-space networks \citep{fontaine2026missing} with arbitrary missing patters.  

Concentration inequalities for classical U-statistics have been extensively studied in the literature. For instance, \cite{hoeffding1948class} derived a Hoeffding-type inequality and \cite{arcones1995bernstein} derived a Bernstein-type inequality. The analyses rely on a trick nowadays known as the Hoeffding decomposition. Without loss of generality, we assume $n=2m$ is even. Let $S^n$ denote all permutations of $1,2,\ldots,n$. Then, 
\[
U^{*}\Big/\binom{n}{2} =\frac{1}{n!}\sum_{\tau\in S^n}Y_{\tau},  \qquad\mbox{with}\;\;  Y_{\tau}:=\frac{h(X_{\tau_1}, X_{\tau_2})+h(X_{\tau_3}, X_{\tau_4})+\ldots + h(X_{\tau_{2m-1}}, X_{\tau_{2m}})}{m}. 
\] 
Here, each $Y_{\tau}$ is a sum of i.i.d. random variables and thus admits a sharp bound for its moment generating function. The normalized U-statistic is a convex combination of $Y_\tau$. By Jensen's inequality, its moment generating function is upper bounded by the convex combination of the moment generating functions of $Y_{\tau}$. 
Combining this with the Chernoff bound gives the desirable concentration inequalities. 

Unfortunately, the Hoeffding decomposition is not applicable to the statistic defined in \eqref{our-U-stat}. We notice that the high-level idea of the Hoeffding decomposition is to partition all pairs $(i,j)$ into non-overlapping clusters such that within each cluster, the corresponding variables $h(X_i, X_j)$ are independent of each other. 
We also notice that the independence requirement is satisfied when any two pairs $(i,j)$ and $(i', j')$ within a cluster share no common node. Inspired by these observations, we connect our problem to the {\it proper edge coloring} problem in graph theory. 

A proper $K$-edge-coloring for a graph refers to an assignment of $K$ colors to all edges so that any two edges of the same color share no common node. Suppose $G$ has a proper $K$-edge-coloring: $E=\cup_{k=1}^K E_k$. Then, for any nonnegative weights $w_1, \ldots, w_K$ such that $\sum_{k=1}^K w_k=0$, we can decompose
\beq \label{new-decomposition}
U = \sum_{k=1}^K w_k Y_k, \qquad\mbox{with}\quad Y_k = \frac{1}{w_k}\sum_{e=(i,j)\in E_k}h_e(X_i, X_j). 
\eeq
Now, each $Y_k$ is a sum of independent random variables, and $U$ is a convex combination of $Y_k$'s. This is analogous to the Hoeffding decomposition for classical U-statistics but broadly applicable to arbitrary graphs, non i.i.d. observations, and asymetric and heterogenous kernels. Using this decomposition, we derive several concentration inequalities for $U$. For example, when each $h_e(X_i, X_j)$ is sub-Gaussian, we show that the tail behavior of $U-\mathbb{E}U$ is similar to ${\cal N}(0, v)$, with 
\beq  \label{sub-Gaussian-norm}
v = K\sum_{e=(i,j)\in E}\mathrm{Var}(h_e(X_i, X_j)). 
\eeq
The concentration bounds are sharper as $K$ becomes smaller. We want to take the smallest $K$ such that a proper $K$-edge-coloring exists. This number is called the chromatic index of $G$. 
By Vizing's theorem \citep{Vizing1964}, the chromatic index of a graph is alway upper bounded by $d_{\max}+1$, where $d_{\max}$ is the maximum node degree. 
Therefore, we can take $K=d_{\max}+1$ in \eqref{sub-Gaussian-norm}, resulting in simple and transparent concentration inequalities for $U$. In the special case where $G$ is a complete graph, our results recover some known concentration inequalities for classical U-statistics.  

\cite{maurer2022exponential} also studied concentration inequalities for statistics like \eqref{our-U-stat} with a deterministic graph $G$.  However, they assumed that  $X_i$'s are i.i.d. and the kernel is homogenous and bounded ($h_e(x, y)=h(x, y)$, and $|h(x,y)|\leq 1$). 
In particular, boundedness of the kernel plays a crucial role in their analysis. 
The resulting concentration inequalities also have more complex and less transparent forms than ours. 

Our edge-coloring approach to deriving concentration inequalities appears to be conceptually simple, and it yields transparent concentration bounds in which the graph enters only through the quantity $d_{\max}$. We will also show that this technique readily extends to more general settings such as incomplete order-$m$ U-statistics and doubly indexed data vectors (see Section~\ref{sec:Ext}). 


\section{New Concentration Inequalities} \label{sec:Main}

Throughout this article, let $N$ denote the total number of edges in $G$, and $d_{\max}$ the maximum node degree. 
In graph theory, a $K$-edge-coloring is an assignment of edges into one of $K$ colors. A proper $K$-edge-coloring is such that no two incident edges (edges that share a common node) have the same color. 
When there exists a proper $K$-edge-coloring for $G$, we have the following theorem: 

\begin{thm}[Main result] \label{thm:main}
Suppose for every $e\in E$, there exists $\sigma_e>0$ such that 
\beq \label{main-condition}
\mathbb{E}\Bigl[e^{\lambda \{h_e(X_i, X_j)-\mathbb{E}h_e(X_i, X_j)\}}\Bigr]\leq e^{\lambda^2\sigma_e^2/2}, \qquad\mbox{for all $\lambda>0$}. 
\eeq 
If there exists a proper $K$-edge-coloring of $G$, then for all $t>0$, 
\[
\mathbb{P}(U-\mathbb{E}U>t)\leq \exp\biggl\{ - \frac{t^2}{2K\sum_{e\in E}\sigma_e^2}\biggr\}. 
\]
\end{thm}

\begin{proof}
Let $E=\cup_{k=1}^K E_k$ be a proper $K$-edge-coloring for $G$. Then, we can write 
\[
U=  \sum_{k=1}^K w_k Y_k, \quad\mbox{where}\;\; Y_k = \frac{1}{w_k}\sum_{e=(i,j)\in E_k}h_e(X_i, X_j), \mbox{ and } w_k = \frac{\sqrt{\sum_{e\in E_k}\sigma^2_e}}{\sum_{k=1}^K\sqrt{\sum_{e\in E_\ell}\sigma^2_e}}. 
\]
By the definition of proper edge coloring, any two edges in $E_k$ share no common node. 
Therefore, $Y_k$ is a sum of $|E_k|$ independent random variables. It follows that for all $\lambda>0$, 
\begin{align*}
\mathbb{E}\Bigl[ &e^{\lambda(Y_k-\mathbb{E}Y_k)}\Bigr]  = \prod_{e=(i,j)\in E_k}\mathbb{E}\Bigl[ e^{\frac{\lambda}{w_k}\{h_e(X_i, X_j)-\mathbb{E}h_e(X_i, X_j)\}}\Bigr] \cr
& \leq \exp\biggl\{ \frac{\lambda^2}{2w_k^2} \sum_{e\in E_k}\sigma_e^2  \biggr\} = \exp\Biggl\{\frac{ \lambda^2}{2} \Biggl(\sum_{\ell=1}^K \sqrt{\sum_{e\in E_\ell}\sigma_e^2}\Biggr)^2 \Biggr\}.
\end{align*}
Using the Cauchy-Schwarz inequality, we further obtain:
\beq \label{main-proof-1}
\mathbb{E}\Bigl[ e^{\lambda(Y_k-\mathbb{E}Y_k)}\Bigr] \leq \exp\biggl\{\frac{\lambda^2}{2}K\sum_{\ell=1}^K\Bigl(\sum_{e\in E_\ell}\sigma_e^2\Bigr)\biggr\} =\exp\biggl\{\frac{\lambda^2K}{2}\sum_{e\in E}\sigma_e^2\biggr\}. 
\eeq
We now apply the Chernoff bound and utilize \eqref{main-proof-1}. It yields that
\begin{align*}
\mathbb{P}(U-\mathbb{E}U>t)&\leq e^{-\lambda t} \, \mathbb{E}\bigl[e^{\lambda(U-\mathbb{E}U)}\bigr] \qquad \mbox{(Chernoff bound)}\cr
&= e^{-\lambda t}\, \mathbb{E}\Bigl[ e^{ \sum_{k=1}^K w_k \lambda (Y_k-\mathbb{E}Y_k)}\Bigr]\cr
&\leq e^{-\lambda t }\sum_{k=1}^K w_k \, \mathbb{E}\bigl[e^{\lambda (Y_k-\mathbb{E}Y_k)}\bigr] \qquad \mbox{(Jensen's inequality)}\cr
& \leq  e^{-\lambda t} \exp\biggl\{\frac{\lambda^2K}{2}\sum_{e\in E}\sigma_e^2\biggr\} \qquad \mbox{(using the inequality \eqref{main-proof-1})}\cr
&= \exp\biggl\{ - \frac{t^2}{2K \sum_{e\in E}\sigma_e^2}\biggr\}, \qquad\mbox{if we take } \lambda=\frac{t}{K\sum_{e\in E}\sigma_e^2}. 
\end{align*}
This gives the desirable claim. 
\end{proof}

The bound in Theorem~\ref{thm:main} holds for every integer $K$ such that 
$G$ admits a proper $K$-edge-coloring. We hope to make $K$ as small as possible. The smallest admissible 
value of $K$ is the chromatic index of 
$G$, denoted by $\chi'(G)$. The Vizing's theorem \citep{Vizing1964} states that for any undirected graph,
\beq \label{Vizing}
d_{\max}\leq \chi'(G)\leq d_{\max}+1, \qquad\mbox{where $d_{\max}$ is the maximum degree of $G$}. 
\eeq
This partitions undirected graphs into two classes: ``class one" graphs for which $d_{\max}$ colors are sufficient for proper edge coloring, and ``class two" graphs for which $d_{\max}+1$ colors are necessary. For our purpose, we use $d_{\max}+1$ as a universal upper bound for the chromatic index. 
It leads to the following corollary: 

\begin{cor}[Sub-Gaussian type inequality] \label{cor:d-max-version}
Suppose \eqref{main-condition} is satisfied, and let $\sigma^2=N^{-1}\sum_{e\in E}\sigma^2_e$. Then, for all $t>0$, 
\[
\mathbb{P}\biggl(\frac{U-\mathbb{E}U}{N}>t\biggr)\leq \exp\biggl\{ - \frac{Nt^2}{2(d_{\max}+1)\sigma^2}\biggr\}.
\]
\end{cor}

We consider two special graphs. The first is the complete graph, where $U=U^{\text{\text{full}}}= \sum_{i<j}h_e(X_i, X_j)$. It further reduces to the classical U-statistics when the kernel is homogeneous and symmetric and when $X_i$'s have identical distributions. Since $N=\frac{n(n-1)}{2}$ and $d_{\max}=n-1$, we have the following corollary: 
\begin{cor}[Complete graph] \label{cor:full-graph}
Suppose \eqref{main-condition} is satisfied, and let $\sigma^2=\frac{2}{n(n-1)}\sum_{1\leq i<j\leq n}\mathrm{Var}(h_e(X_i, X_j))$. Consider $U^{\text{\text{full}}}= \sum_{1\leq i<j\leq n}h_e(X_i, X_j)$. Then, for all $t>0$,
\[
\mathbb{P}\biggl(\frac{U^{\text{full}}-\mathbb{E}U^{\text{full}}}{\frac{n(n-1)}{2}}>t\biggr)\leq \exp\biggl\{ - \frac{(n-1)t^2}{4\sigma^2}\biggr\}.
\]
\end{cor}

Corollary~\ref{cor:full-graph} provides a sub-Gaussian-counterpart for some well-known concentration inequalities for classical U-statistics \citep{hoeffding1948class, arcones1995bernstein, pitcan2017note}. Meanwhile, Corollary~3 is more general, allowing for non i.i.d. observations as well as heterogeneous and asymmetric kernels.


The second special case is when $G$ is an Erd\H{o}s--R\'enyi random graph with edge probability $q_n\in (0,1)$. Then, $U$ can be regarded as a variant of the incomplete U-statistics \citep{blom1976incomplete}, allowing heterogeneous and asymmetric kernels and non-identically distributed $X_i$. 
We note that the incomplete U-statistics are used to estimate $\mathbb{E}[h(X_1, X_2)]$, which is also equal to $N^{-1}\mathbb{E}[U|G]$, regardless of the realization of $G$. 
Inspired by this observation, we are interested in the following quantity:
\[
\delta(G):= \frac{U-\mathbb{E}[U|G]}{N}, \qquad\mbox{where $N=N(G)$ is the total number of edges in $G$}. 
\]
By Corollary~\ref{cor:d-max-version}, conditioning on $G$, $
\delta(G) \leq  N^{-\frac12}\sigma \sqrt{2(d_{\max}+1)\log(n)}$, with probability $1-n^{-1}$. 
Since the upper bound for $\delta(G)$ only depends on $N$ and $d_{\max}$, we then study these two quantities under an Erd\H{o}s--R\'enyi model. 
Note that $N=\frac{1}{2}\sum_{1\leq i\leq n}d_i$ and $d_{\max}=\max_{1\leq i\leq n}d_i$, where $d_i$ is the degree of node $i$. 
The most interesting asymptotic regime is when $\omega_n:=nq_n/\log(n)\to\infty$. By Corollary 33.7 of \cite{frieze2015introduction} (see also Theorem 3.4 there), 
\[
\mathbb{P}\biggl(\cup_{1\leq i\leq n}\Bigl\{|d_i- nq_n|> \omega_n^{-1/3}nq_n\Bigr\} \biggr)\leq 2n^{-(1/3) \omega_n^{1/3}} = o(n^{-1}). 
\]
It immediately leads to a high-probability upper bound for $d_{\max}$ and a high-probability lower-bound for $N$. We then have the following corollary: 

\begin{cor}[Erd\H{o}s--R\'enyi graph]
Suppose  \eqref{main-condition} is satisfied and $G$ is an Erd\H{o}s--R\'enyi random graph with edge probability $q_n\in (0,1)$. Suppose $\omega_n: = nq_n/\log(n)\to\infty$. Let $\beta_n = \frac{(1+\omega_n^{-1/3})nq_n+1}{(1-\omega_n^{-1/3})nq_n}$. Then, with probability $1-O(n^{-1})$, 
\[
\delta(G)\leq 2\sigma n^{-\frac12}\sqrt{\beta_n\log(n)}.  
\]
\end{cor}


The above results are based on the sub-Gaussian assumption in \eqref{main-condition}. 
We now consider the cases where $h_e(X_i, X_j)$ are bounded, and
develop a Hoeffding-type inequality and a Bernstein-type inequality. 

First, suppose $a_e\leq h_e(X_i, X_j)\leq b_e$. The Hoeffding's lemma implies $\mathbb{E}[e^{\lambda\{h_e(X_i, X_j)-\mathbb{E}h_e(X_i, X_j)\}}]\leq e^{\lambda^2(b_e-a_e)^2/8}$, which means \eqref{main-condition} holds with $\sigma_e^2=(b_e-a_e)^2/4$. Combining this observation with Corollary~\ref{cor:d-max-version} yields the following result: 

\begin{cor}[Hoeffding-type]
Suppose $a_e\leq h_e(X_i, X_j)\leq b_e$ for all $e\in E$. For all $t>0$,
\[
\mathbb{P}\biggl(\frac{U-\mathbb{E}U}{N}>t\biggr)\leq \exp\biggl\{ - \frac{2N^2t^2}{(d_{\max}+1)\sum_{e\in E}(b_e-a_e)^2}\biggr\}.
\]
\end{cor}

Next, we derive a Bernstein-type inequality. The concentration inequalities established above require only the existence of a proper edge coloring. To obtain a Bernstein-type inequality, however, we additionally need a global size-balance property of edge colorings \citep{FolkmanFulkerson1969,asratian2000some}. 
Let $N_1, N_2, \ldots, N_K$ be the total number of edges assigned to each color. It is known that (e.g., see Corollary 1.2 of \cite{asratian2000some}) for all $K\geq \chi'(G)$, there exists a proper $K$-edge-coloring such that 
\beq \label{global-size-balance}
\max_{1\leq k\neq \ell\leq K}|N_k-N_\ell|\leq 1 \qquad \mbox{(i.e., the color class sizes only differ by at most 1)}. 
\eeq
We leverage \eqref{global-size-balance} to obtain to the following theorem:

\begin{thm}[Berstein-type]
Suppose $|h_e(X_i, X_j)-\mathbb{E}h_e(X_i, X_j)| \leq b$ and $\mathrm{Var}(h_e(X_i, X_j))\leq \sigma^2$ for all $e\in E$. Let $N_*=K\lceil N/K\rceil$. Then, for all $t>0$,
\[
\mathbb{P}\Bigl(\frac{U-\mathbb{E}U}{N_*}>t\Bigr)\leq  \exp\biggl\{ -\frac{N_*t^2}{2(d_{\max}+1)(\sigma^2 +bt/3)}\biggr\}.
\]
\end{thm}
\begin{proof}
We assume $b=1$ without loss of generality (otherwise we can re-scale $h_e(\cdot,\cdot)$ to $b^{-1}h_e(\cdot,\cdot)$ and replace $(\sigma^2, t)$ by $(\sigma^2/b^2, t/b)$, so that it reduces to the case of $b=1$). 

Let $K=d_{\max}+1$. By \eqref{Vizing}, $K\geq \chi'(G)$. It follows from \eqref{global-size-balance} that there exists a proper $K$-edge-coloring, $E=\cup_{k=1}^K E_k$, such that $\max_{k\neq \ell}||E_k|-|E_\ell||\leq 1$. As a result, 
\[
\max_k |E_k|\leq M:= \lceil N/K\rceil.
\]
Write $U=\sum_{k=1}^K\frac{1}{K}Y_k$, where $Y_k=K\sum_{e=(i,j)\in E_k}h_e(X_i, X_j)$. Note that $Y_k$ is a sum of independent random variables. For any random variables $Z$ such that $|Z|\leq 1$ and $\mathrm{Var}(Z)\leq \sigma^2$, it is well-known that
\[
\mathbb{E}\Bigl[e^{\lambda (Z-\mathbb{E}Z)}\Bigr]\leq e^{\sigma^2\phi(\lambda)}, \qquad\mbox{where}\quad \phi(\lambda)=e^{\lambda}-\lambda-1. 
\]
As a result, 
\[
\mathbb{E}\Bigl[e^{\lambda(Y_k-\mathbb{E}Y_k)}\Bigr]=\prod_{e\in E_k} \mathbb{E}\Bigl[e^{\lambda K \{h_e(X_i, Y_i)-\mathbb{E}h_e(X_i, Y_i)\}}\Bigr]\leq \prod_{e\in E_k} e^{\sigma^2\phi(\lambda K)} \leq e^{M \sigma^2\phi(\lambda K)}. 
\]
Using the Chernoff bound, we obtain:
\begin{align*}
\mathbb{P}(U-\mathbb{E}U>t)&\leq e^{-\lambda t}\mathbb{E}\bigl[e^{\lambda (U-\mathbb{E}U)}\bigr] = e^{-\lambda t}\mathbb{E}\Bigl[e^{\sum_{k=1}^K \frac{1}{K}\lambda (Y_k-\mathbb{E}Y_k)}\Bigr]\cr
&\leq e^{-\lambda t}\frac{1}{K}\sum_{k=1}^K \mathbb{E}\bigl[e^{\lambda(Y_k-\mathbb{E}Y_k)}\bigr] \qquad\mbox{(by Jensen's inequality)}\cr
&\leq e^{-\lambda t + M\sigma^2\phi(\lambda K)}. 
\end{align*}
Let $F_t(\lambda)=-\lambda t + M\sigma^2\phi(\lambda K)$. Letting $F_t'(\lambda)=0$ gives $\lambda= \frac{1}{K}\log(1+\frac{t}{MK\sigma^2})$. With this choice of $\lambda$, the above bound becomes
\[
\mathbb{P}(U-\mathbb{E}U>t) \leq \exp\biggl\{ -M\sigma^2\Bigl[ \Bigl(1+\frac{t}{MK\sigma^2}\Bigr)\log\Bigl(1+\frac{t}{MK\sigma^2}\Bigr)-\frac{t}{MK\sigma^2}\Bigr]\biggr\}. 
\]
It is known that $(1+x)\log(1+x)-x\geq \frac{x^2}{2(1+x/3)}$. As a result, 
\[
\mathbb{P}(U-\mathbb{E}U>t) \leq \exp\biggl\{ -\frac{t^2}{2K(MK\sigma^2+t/3)}\biggr\}\leq \exp\biggl\{ -\frac{t^2}{2K(N_*\sigma^2 +t/3)}\biggr\}. 
\]
It follows immediately that
\[
\mathbb{P}\Bigl(\frac{U-\mathbb{E}U}{N_*}>t\Bigr)\leq  \exp\biggl\{ -\frac{N_*t^2}{2K(\sigma^2 +t/3)}\biggr\}.
\]
This proves the claim for $b=1$. 
\end{proof}

\section{Extensions} \label{sec:Ext}
Our proof strategy is readily adaptable to more general settings. To illustrate its versatility, we present two such extensions in this section: one concerning order-$m$ incomplete U-statistics and the other concerning doubly indexed observations.

\subsection{Order-$m$ incomplete U-statistics} \label{subsec:tensor}
Fix $m\geq 2$. We consider an $m$-uniform hypergraph $H=(V,E)$, where the node set is $V=\{1,2,\ldots,n\}$, and each hyper-edge $e\in E$ is a subset of $m$ nodes. Let $\{h_e\}_{e\in E}$ be a collection of $m$-variate kernels indexed by the edges in $H$. Same as before, let $X_1, X_2, \ldots, X_n\in\mathbb{R}^p$ be independent but not necessarily identically distributed random vectors. Define
\beq \label{U-order-m}
U = \sum_{e=(i_1, i_2,\ldots, i_m)\in E}h_e(X_{i_1}, X_{i_2}, \ldots, X_{i_m}). 
\eeq

A proper $K$-edge-coloring of $H$ is a partition of hyper-edges into $K$ non-overlapping groups such that any two hyper-edges within the same group share no common node. The smallest $K$ such that proper $K$-edge-coloring exists is called the chromatic index of $H$ and denoted by $\chi'(H)$. 

The proof of the following theorem is similar to that of Theorem~\ref{thm:main} and thus omitted. 
\begin{thm}[General $m$]\label{thm:main-order-m}
Suppose for every $e=(i_1, i_2,\ldots,i_m)\in E$, there exists $\sigma_e>0$ such that 
\beq \label{main-cond-order-m}
\mathbb{E}\Bigl[e^{\lambda \{h_e(X_{i_1}, \ldots, X_{i_m})-\mathbb{E}h_e(X_{i_1}, \ldots, X_{i_m})\}}\Bigr]\leq e^{\lambda^2\sigma_e^2/2}, \qquad\mbox{for all $\lambda>0$}. 
\eeq
If there exists a proper $K$-edge-coloring of $G$, then for all $t>0$, 
\[
\mathbb{P}(U-\mathbb{E}U>t)\leq \exp\biggl\{ - \frac{t^2}{2K\sum_{e\in E}\sigma_e^2}\biggr\}. 
\]
\end{thm}

When applying Theorem~\ref{thm:main-order-m}, we choose $K$ to be the chromatic index $\chi'(H)$ or its upper bound.
Below, we consider three examples.  

In the first example, $H$ is a complete hyper-graph. Then, $U=U^{\text{\text{full}}}= \sum_{i_1<i_2<\ldots<i_m} h_e(X_{i_1}, \ldots, X_{i_m})$. 
The chromatic index for complete hyper-graphs has an exact form \citep{baranyai1974factrization}
\[
\chi'(H)=\left\lceil {n\choose m}/\lfloor n/m\rfloor \right\rceil.
\]
Let $\sigma^2=\frac{1}{{n\choose m}}\sum_{i_1< i_2<\ldots< i_m}\mathrm{Var}(h_e(X_{i_1}, \ldots, X_{i_m}))$. It follows by Theorem~\ref{thm:main-order-m} that for all $t>0$, 
\beq \label{hypergraph-example-1}
\mathbb{P}\biggl(\frac{U^{\text{full}}-\mathbb{E}U^{\text{full}}}{{n\choose m}}>t\biggr)\leq \exp\biggl\{ - \frac{\lfloor n/m\rfloor t^2}{2\sigma^2}\biggr\}. 
\eeq
We note that $U^{\text{full}}$ is a generalization of the classical U-statistics allowing non-identically distributed $X_i$ and heterogeneous and asymmtric kernels $h_e$.

In the second example, $H$ is an $r$-regular hypergraph, where $r$ is such that $N = nr/m$ is an integer.  By definition of $r$-regular hypergraphs,  each node of $H$ belongs to exactly $r$ hyper-edges. In this case, the total number of hyper-edges is equal to $N$, and the chromatic index satisfies that 
\[
\chi'(H)\leq m(r-1)+1.
\] 
Let $\sigma^2=\frac{1}{N}\sum_{e\in E}\mathrm{Var}(h_e(X_{i_1}, \ldots, X_{i_m}))$. It is implied by Theorem~\ref{thm:main-order-m} that 
\beq \label{hypergraph-example-2}
\mathbb{P}\biggl(\frac{U-\mathbb{E}U}{N}>t\biggr)\leq \exp\biggl\{ - \frac{n t^2}{2m^2(1-\frac{m-1}{mr})\sigma^2}\biggr\}.
\eeq

In the third example, we consider an arbitrary hypergraph. For each node $i$, let $d_i$ be its degree, which is defined as the total number of edges containing node $i$. Let $d_{\max}$ be the maximum node degree. Two hyper-edges are said to be {\it adjacent} if they share at least one common node. We consider the following edge-coloring algorithm. First, list all hyper-edges in an arbitrary order. Then, process them sequentially. When coloring a hyper-edge $e$, exclude the colors already assigned to all previously colored hyper-edges adjacent to $e$, and assign $e$ any remaining color. Clearly, if every hyper-edge is adjacent to at most $K_0$ hyper-edges, then the algorithm succeeds using $K = K_0+1$ colors. By definition, each node of a hyper-edge $e$ belongs to at most $d_{\max}-1$ other hyper-edges. Since $e$ contains $m$ nodes, it is adjacent to at most $K_0=m(d_{\max}-1)$ hyper-edges. Consequently, 
\[
\chi'(H)\le m(d_{\max}-1)+1.
\]
This universal upper bound of $\chi'(H)$ can be combined with Theorem~\ref{thm:main-order-m} to obtain a concentration inequality applicable to any hypergraph $H$. 
For example, when hyper-edges are randomly generated with probability $q_n$, the maximum degree is of the order $n^{m-1} q_n$. Then, there exists a constant $C_m>0$ such that for all sufficiently large $n$ and all $t>0$, 
\beq \label{hypergraph-example-3}
\mathbb{P}(U-\mathbb{E}U>t)\leq \exp\biggl\{ - \frac{t^2}{C_m n^{m-1} q_n\sum_{e\in E}\sigma_e^2}\biggr\}. 
\eeq

In discrete and combinatorial mathematics, there are results about the chromatic index of certain types of hypergraphs. For instance, when the hypergraph is linear (i.e., any two hyper-edges share at most one common node), it is conjectured \citep{berge1989chromatic,furedi1986chromatic} and partially proved that $\chi'(H)\leq d_{\max}+1$. Users are free to plug into Theorem~\ref{thm:main-order-m} any upper bound of $\chi'(H)$ tailored to the hypergraphs arising in their applications. This flexibility provides a convenient and practical approach to deriving concentration inequalities for these kinds of statistics.

\subsection{Incomplete U-statistics on doubly indexed random vectors} \label{subsec:doubleIndex}
Consider a collection of random vectors $X_{i\ell}\in\mathbb{R}^p$, for $1\leq i\leq n$ and $1\leq \ell\leq L_i$, which are independent but not necessarily identically distributed. Let $G=(V,E)$ be an undirected graph on $\{1,2,\ldots,n\}$, and let $h_e(\cdot, \cdot)$ denote a bivariate kernel, for $e\in E$. For each $1\leq i\leq n$, let $G_i=(V_i, E_i)$ be an undirected graph on $\{1,2,\ldots, L_i\}$, and let $h_{i,e}(\cdot, \cdot)$ denote a bivariate kernel indexed by $i$ and $e\in E_i$. Define
\beq \label{our-U-stat-double}
U_1=\sum_{e=(i,j)\in E}\sum_{\ell=1}^{L_i}\sum_{k=1}^{L_j}h_e(X_{i\ell}, X_{jk}), \qquad U_2=\sum_{i=1}^n \sum_{e=(\ell,k)\in E_i} h_{i, e}(X_{i\ell}, X_{ik}). 
\eeq
They extend the statistic in \eqref{our-U-stat} to incorporate doubly indexed random vectors.  

A motivating application of $U_1$ arises in the Bradley-Terry model with covariates \citep{turner2012bradley}. Suppose there are $n$ entities (e.g., journals), and let $G$ be a fixed comparison graph. For each pair $(i,j)$, entity $i$ has $L_i$ samples (e.g., papers), entity $j$ has $L_j$ samples, and  $h_e(X_{i\ell}, X_{jk})$ represents the interaction between a sample from entity $i$ and a sample of from entity $j$ (e.g., a citation between one paper to another).  
For $U_2$, a motivating application arises in the Poisson-process topic model \citep{austern2025poisson}. Here, each $1\leq i\leq n$ corresponds to a document,  $L_i$ is the document length, and $X_{i1}, \ldots, X_{iL_i}$ are the word embeddings. The function $h_e(X_{i\ell}, X_{ik})$ can then represent the attention between two words within the document.

Theorem~\ref{thm:main} can be extended to these settings. Note that in the proof of this theorem, we obtain not only a tail probability bound but also an upper bound for the moment generating function: $\mathbb{E}[e^{\lambda(U-\mathbb{E}U)}]\leq \exp\bigl\{\frac{\lambda^2K}{2}\sum_{e\in E}\sigma_e^2\bigr\}$. 
Using this bound, we can easily obtain the following corollaries: 

\begin{cor} \label{thm:doubleIndex-1}
Suppose $\mathbb{E}\bigl[e^{\lambda \{h_e(X_{i\ell}, X_{jk})-\mathbb{E}h_e(X_{i\ell}, X_{jk})\}}\bigr]\leq e^{\lambda^2\sigma_e^2/2}$, for all $\lambda>0$. Let $d_{\max}$ denote the maximum node degree of $G$. Then, for all $t>0$, 
\[
\mathbb{P}(U_1-\mathbb{E}U_1>t)\leq \exp\biggl\{ - \frac{t^2}{2(d_{\max}+1)\sum_{e=(i,j)\in E}L_iL_j(L_i\vee L_j)\sigma_e^2}\biggr\}. 
\]
\end{cor}

\begin{proof}
Let ${\boldsymbol X}_i\in\mathbb{R}^{pL_i}$ be the vector by stacking $X_{i1}, X_{i2}, \ldots, X_{iL_i}$ together. Then, we can re-write 
\[
U_1 = \sum_{e=(i,j)\in E}\tilde{h}_e({\boldsymbol X}_i, {\boldsymbol X}_j), \qquad \mbox{with}\quad \tilde{h}_e({\boldsymbol X}_i, {\boldsymbol X}_j):= \sum_{\ell=1}^{L_i}\sum_{k=1}^{L_j}h_e(X_{i\ell}, X_{jk}). 
\]
We first study $\tilde{h}_e({\boldsymbol X}_i, {\boldsymbol X}_j)$ for a fixed $e=(i,j)$. Consider a complete bipartite graph $G_{ij}$,  which has $L_i$ type-1 nodes and $L_j$ type-2 nodes, as well as an edge between each pair of type-1 node and type-2 node. Then, $\tilde{h}_e({\boldsymbol X}_i, {\boldsymbol X}_j)=\sum_{(\ell, k)\in G_{ij}}h_e(X_{i\ell}, X_{jk})$, which is also a U-statistic. Following similar steps in the proof of Theorem~\ref{thm:main}, we obtain that, when $G_{ij}$ has a proper $K$-edge coloring, for all $\lambda >0$, 
\[
\mathbb{E}\Bigl[e^{\lambda\{ \tilde{h}_e({\boldsymbol X}_i, {\boldsymbol X}_j) - \mathbb{E}\tilde{h}_e({\boldsymbol X}_i, {\boldsymbol X}_j) \}}\Bigr]\leq \exp\biggl\{\frac{\lambda^2K}{2}\sum_{(\ell,k)\in G_{ij}}\sigma_e^2\biggr\}= \exp\biggl\{\frac{\lambda^2KL_iL_j\sigma_e^2}{2}\biggr\}. 
\]
The chromatic index for the complete bipartite graph $G_{ij}$ is $\max\{L_i, L_j\}$, so we can take $K=\max\{L_i, L_j\}$ in the above inequality. It implies that $\tilde{h}_e({\boldsymbol X}_i, {\boldsymbol X}_j)$ satisfies \eqref{main-condition} with $\tilde{\sigma}_e^2 = L_iL_j(L_i\vee L_j)\sigma_e^2$. Then, we can apply Theorem~\ref{thm:main} to $\sum_{e=(i,j)\in E}\tilde{h}_e({\boldsymbol X}_i, {\boldsymbol X}_j)$ to obtain the claim. 
\end{proof}

\begin{cor} \label{thm:doubleIndex-2}
Suppose $\mathbb{E}\bigl[e^{\lambda \{h_{i,e}(X_{i\ell}, X_{ik})-\mathbb{E}h_{i,e}(X_{i\ell}, X_{ik})\}}\bigr]\leq e^{\lambda^2\sigma_{i,e}^2/2}$, for all $\lambda>0$. Let $d_{i,\max}$ denote the maximum node degree of $G_i$. Then, for all $t>0$, 
\[
\mathbb{P}(U_2-\mathbb{E}U_2>t)\leq \exp\biggl\{ - \frac{t^2}{2\sum_{i=1}^n(d_{i,\max}+1)\sum_{e\in E_i}\sigma_{i,e}^2}\biggr\}. 
\]
\end{cor}
\begin{proof}
Write $U_2=\sum_{i=1}^n V_i$, where $V_i=\sum_{e=(\ell,k)\in E_i}h_{i,e}(X_{i\ell}, X_{ik})$. Each $V_i$ has the same form as in \eqref{our-U-stat}; and by \eqref{Vizing}, $G_i$ has a proper $K$-edge-coloring for $K=d_{i,\max}+1$. As in the proof of Theorem~\ref{thm:main}, we can show that 
\[
\mathbb{E}\bigl[e^{\lambda(V_i-\mathbb{E}V_i)}\bigr]\leq \exp\biggl\{\frac{\lambda^2(d_{i,\max}+1)}{2} \sum_{e\in E_i}\sigma_{i,e}^2\biggr\}.
\] 
Using the Chernoff bound, $\mathbb{P}(U_2-\mathbb{E}U_2>t)\leq e^{-\lambda t}\prod_{i=1}^n \mathbb{E}[e^{\lambda (V_i-\mathbb{E}V_i)}]$. The claim follows by plugging in the above inequality and optimizing over $\lambda$. 
\end{proof}

\section{Discussion} \label{sec:Discuss}
The main contribution of this work is a delicate connection between graph coloring and concentration inequalities of U-statistics. It 
allows us to handle an arbitrary sampling graph, while the classical Hoeffding decomposition cannot deal with this case. 

One limitation of this work is that the variance quantity appearing in our concentration inequalities is the marginal variance, $\mathrm{Var}(h_e(X_i,X_j))$, which may not be optimal. For example, in the classical U-statistic setting, where $G$ is the complete graph and $h_e=h$ for all edges, \cite{hoeffding1948class} showed that the natural variance quantity is instead $\mathrm{Var}(\mathbb{E}[h(X_i,X_j)| X_i])$, which is generally smaller. More recently, \cite{maurer2022exponential} established a Bernstein-type inequality for U-statistics associated with an arbitrary graph $G$, in which the variance term is also based on $\mathrm{Var}(\mathbb{E}[h(X_i,X_j)| X_i])$.


In many applications, these two variance quantities differ only by a constant factor and therefore lead to the same asymptotic behavior. Moreover, although the above results are sharper in certain special settings---for example, when the kernel is uniformly bounded and homogeneous---their proofs rely heavily on these assumptions and are hard to extend to more general scenarios. In contrast, our graph-coloring approach is more flexible: it readily accommodates unbounded kernels and arbitrary sampling graphs, while yielding concise and transparent proofs.


We also note that our results are intended for non-degenerate U-statistics, that is, 
$\mathbb{E}[h_e(X_i,X_j)| X_i]\neq 0$ or $
\mathbb{E}[h_e(X_i,X_j)| X_j]\neq 0$. 
For classical degenerate U-statistics, where $G$ is the complete graph and the conditional expectations vanish, decoupling inequalities \citep{de2012decoupling} are the standard analytical tool. Extending our graph-coloring proof idea to degenerate U-statistics associated with arbitrary sampling graphs remains an interesting direction for future work.

\bibliography{Ustat}
\bibliographystyle{chicago}

\end{document}